\documentclass{amsart}
\usepackage{amsbsy}
\usepackage{amsfonts}
\usepackage{amsmath}
\usepackage{amsthm}
\usepackage{amssymb}
\usepackage{enumitem}
\usepackage{graphicx}
\usepackage{hyperref}
\begin{document}

\title{Macroscopic Pair Correlation of the Riemann Zeroes for Smooth Test Functions}
\author{Brad Rodgers}
\address{Department of Mathematics, UCLA, Los Angeles CA 90095-1555, USA}
\email{brodgers@math.ucla.edu}
\maketitle

\begin{abstract}
    On the assumption of the Riemann hypothesis, we show that over a class of sufficiently smooth test functions, a measure conjectured by Bogomolny and Keating coincides to a very small error with the actual pair correlation measure for zeroes of the Riemann zeta function. Our result extends the well known result of Montgomery that over the same class of test functions the pair correlation measure coincides (to a larger error term) with that of the Gaussian Unitary Ensemble (GUE). The restriction of test functions remains stringent, but we are nonetheless able to detect, at a microscopically blurred resolution, macroscopic troughs in the pair correlation measure.
\end{abstract}

\newenvironment{nmath}{\begin{center}\begin{math}}{\end{math}\end{center}}

\newtheorem{thm}{Theorem}[section]
\newtheorem{lem}[thm]{Lemma}
\newtheorem{prop}[thm]{Proposition}
\newtheorem{cor}[thm]{Corollary}
\newtheorem{conj}[thm]{Conjecture}
\newtheorem{dfn}[thm]{Definition}

\newcommand\MOD{\textrm{ (mod }}
\newcommand\s{\textrm{ }}

\section{Introduction}

This note is an attempt to see how far the macroscopic statistics of the Riemann zeta function can be understood in a rigorous fashion. By this we mean especially those numerical statistics which seem to indicate a statistical repulsion between the low lying zeroes of Riemann's zeta function and the differences of these zeroes. (See Figure \ref{data}.) Such statistics were first noticed by Bogomolny and Keating \cite{BogKeat}, who predicted them heuristically, and recently rediscovered by P\'{e}rez-Marco in \cite{Perez}. Snaith's paper \cite{Snaith} contains a relatively recent survey.

We will make one concession of rigor, which is to assume from this point that the Riemann Hypothesis is true. This is not entirely necessary, but without it the results (and conjectures) that follow would not be particularly meaningful.\footnote{If for a sufficiently nice function $f$ one understands $f(x+iy)$ in the harmonic sense to be $\int \hat{f}(\xi)e((x+iy)\xi)d\xi$, then what follows can made unconditional. This observation, which has been made before for similar problems, would seem to be of rather secondary interest.} We will show that a formula first conjectured by Bogomolny and Keating indicating the observed repulsion is true for sufficiently smooth test functions. Our approach in some ways consists in nothing more than carefully coloring in their heuristic computations.

We will follow the conventions that $e(x) = e^{i2\pi x}$, $\hat{f}(\xi) = \int e(-x \cdot \xi) f(x) dx$. In addition we use the notation $f(x) \lesssim g(x)$ in place of $f(x) \ll g(x)$.

If the non-trivial zeroes of $\zeta(s)$ are written in the form $1/2+i\gamma$, then the Riemann Hypothesis is the statement that $\gamma$ is always real. We will slightly abuse terminology by sometimes referring to the $\gamma$'s as `zeroes'; this should cause no confusion. The number of such $\gamma$ in an interval $[T\,,\,T+1]$ is known to be roughly $\tfrac{\log T}{2\pi}$, so that the spacing between consecutive zeroes in this interval is on average $\tfrac{2\pi}{\log T}$. (A slightly better approximation to density near $T$ is $\tfrac{\log (T/2\pi)}{2\pi}$, which will make an appearance later.)

In 1973 Montgomery \cite{Montgomery} made the following more precise conjecture concerning these spacings
\begin{conj}[Pair Correlation Conjecture]\label{PCC}
For a fixed Schwartz class test function $f$,
$$
\sum_{\substack{\gamma, \gamma' \in [0,T]\\\mathrm{distinct}}} f\big(\tfrac{\log T}{2\pi}(\gamma-\gamma')\big) = T\frac{\log T}{2\pi} \bigg(\int_{-\infty}^\infty f(u) \bigg[1-\Big(\frac{\sin \pi u}{\pi u}\Big)^2\bigg] du + o(1)\bigg).
$$
\end{conj}
\noindent More informally, for a random $\gamma'$ of height roughly $T$, the expected number of distinct $\gamma$ to lie in an interval $[\gamma' + \tfrac{2\pi\alpha}{\log T}\,,\,\gamma'+\tfrac{2\pi\beta}{\log T}]$ is $\int_\alpha^\beta 1-(\tfrac{\sin \pi u}{\pi u})^2 du$. Since the integrand is small when $u$ is small, there is very little chance that the distance between zeroes will be orders less than $1/\log T$, and in this sense zeroes repel one another.
\begin{figure}
  \centering
    \includegraphics{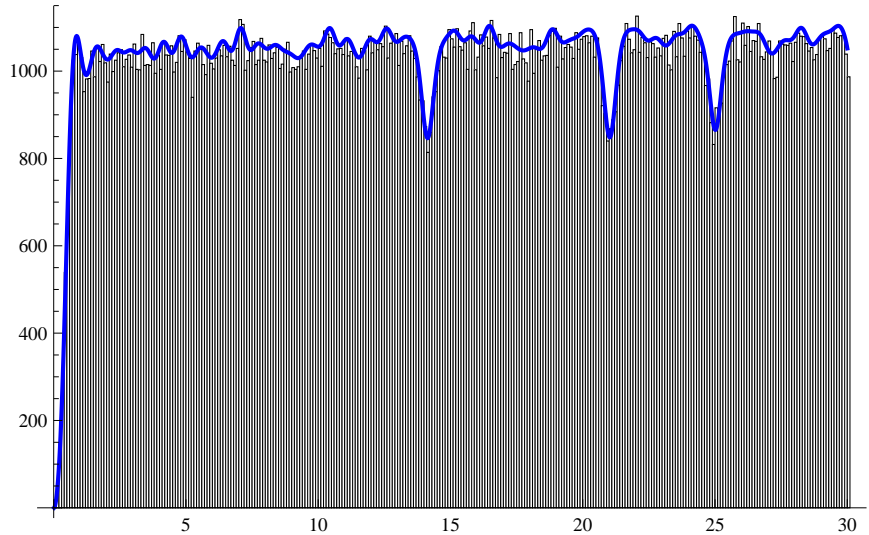}
  \caption{A histogram of $\gamma-\gamma'$ for the first 10,000 zeroes, counting the number of such differences to lie in intervals of size 0.1, compared to the appropriately scaled Bogomolny-Keating prediction in Theorem \ref{4corr}. Note that the smaller troughs around 14.13, 21.02, and 25.01 occur approximately at locations of zeroes themselves (although this pattern becomes less prominent around higher-lying zeroes; see the discussion at the end of section 4), and that the larger trough at the origin would, if stretched out by a factor of $\log T/2\pi$, resemble the GUE measure in Montgomery's conjecture, where $T$ is the number of zeroes included in the histogram. Generated with Mathematica.}
  \label{data}
\end{figure}
This conjecture Montgomery derived from a slightly stronger conjecture:
\begin{conj}[Strong Pair Correlation Conjecture]
\label{SPC}
For fixed $M$ and $w(u) = \tfrac{4}{4+u^2}$,
\begin{equation}
\label{strongpair}
\frac{2\pi}{T \log T} \sum_{0\leq \gamma, \gamma' \leq T} T^{i\alpha (\gamma-\gamma')} w(\gamma-\gamma') = 1 - (1-|\alpha|)_+ + o(1) + (1+o(1))T^{-2\alpha}\log T,
\end{equation}
uniformly for $\alpha \in [-M\,,\, M]$.
\end{conj}
\noindent Note that here $\gamma, \gamma'$ need not be distinct, and for $\alpha$ away from $0$, the term $T^{-2\alpha}\log T$ is unimportant.

For any $\epsilon > 0$, he was able to prove this for $M = 1-\epsilon$, and by integration in $\alpha$ he could prove Conjecture \ref{PCC} for $f$ with $\textrm{supp} \hat{f} \subset [-M\,,\,M]$. That this holds uniformly across $\alpha$ and that the weight function $w(\gamma-\gamma')$ collects $\gamma$ in the vicinity of $\gamma'$ numbering not just $O(1)$ but  $O(\log T)$  has a certain statistical significance which is not typically remarked upon but which extends beyond Conjecture \ref{PCC}.

We may draw out the point -- and motivate our computations that follow -- by putting the Strong Pair Correlation Conjecture in a slightly different form. Montgomery made use of the fact that the left hand side of \eqref{strongpair} is equal to
\begin{align*}
&\frac{4}{T\log T} \int_0^T \bigg|\sum_\gamma \frac{T^{i\alpha\gamma}}{1+(\gamma-t)^2}\bigg|^2\,dt\,+O\Big(\frac{\log ^2 T}{T}\Big) \\
&\;= \frac{4}{T\log T} \int_0^T \sum_{\gamma,\gamma'} \rho(\gamma-t)\rho(\gamma'-t)e\big(\alpha\tfrac{\log T}{2\pi}(\gamma-t) - \alpha \tfrac{\log T}{2\pi}(\gamma'-t)\big)\,dt + O\Big(\frac{\log ^2 T}{T}\Big),
\end{align*}
where $\rho(u) = 1/(1+u^2)$. The passage to the first line is made possible by knowing what is in effect the one-level density of Zeta zeroes, that $$
N(T):= \#\{\gamma\in (0,T)\} = \tfrac{T}{2\pi}\log\big(\tfrac{T}{2\pi}\big)-\tfrac{T}{2\pi}+O(\log T).
$$
The reader is referred to Montgomery's paper for details. On the other hand, recalling the Fourier pair,
$$
g(\nu):= \Big(\frac{\sin \pi \nu}{\pi \nu}\Big)^2
$$
$$
\hat{g}(x) = (1-|x|)_+,
$$
we have by Parseval,
\begin{align*}
&\int_{\mathbb{R}^2} \rho\big(\tfrac{2\pi \nu_1}{\log T}\big)\rho\big(\tfrac{2\pi \nu_1}{\log T}\big)e\big(\alpha\nu_1-\alpha\nu_2)\bigg[1+\delta(\nu_1-\nu_2)-\Big(\frac{\sin \pi (\nu_1-\nu_2)}{\pi (\nu_1-\nu_2)}\Big)^2 \bigg]d\nu_1d\nu_2\\
&\;=  \frac{\log ^2 T}{4\pi^2}\int_{\mathbb{R}^2} \hat{\rho}\big(-\tfrac{\log T}{2\pi}\xi_1\big) \hat{\rho}\big(-\tfrac{\log T}{2\pi}\xi_2\big)\delta(\xi_1+\xi_2) \big[1+\delta\big(\tfrac{\xi_1-\xi_2}{2}+\alpha\big)-\big(1-\big|\tfrac{\xi_1-\xi_2}{2}+\alpha\big|\big)_+\big]\,d\xi_1d\xi_2\\
&\;= \frac{\log^2 T}{4} \int_\mathbb{R} e^{-2\log T |\xi|}\big[1+\delta(\xi+\alpha)-(1-|\xi+\alpha|)_+\big]\,d\xi\\
&\;= \frac{\log^2 T}{4}\Big(T^{-2\alpha}+ \tfrac{1}{\log T}[1-(1-|\alpha|)_++o(1)]\Big)
\end{align*}
We may conclude that Montgomery's Strong Pair Correlation Conjecture is equivalent to the conjecture that
\begin{align*}
&\frac{1}{T}\int_0^T \sum_{\gamma, \gamma'} \rho(\gamma-t)\rho(\gamma'-t) e\big(\alpha\tfrac{\log T}{2\pi} (\gamma-t) - \alpha\tfrac{\log T}{2\pi}(\gamma'-t)\big)\,dt\\
&\; = (1+o(1))\int_{\mathbb{R}^2} \rho\big(\tfrac{2\pi\nu_1}{\log T}\big)\rho\big(\tfrac{2\pi\nu_2}{\log T}\big) e\big(\alpha(\nu_1-\nu_2)\big)\bigg[1+\delta(\nu_1-\nu_2)-\Big(\frac{\sin \pi (\nu_1-\nu_2)}{\pi (\nu_1-\nu_2)}\Big)^2\bigg]\,d\nu_1d\nu_2.
\end{align*}
Here the $\delta$ function corresponds with those terms on the left in which $\gamma=\gamma'$ and both can be removed accordingly. In fact, this is equivalent to the claim that
\begin{align}
\label{SPC2}
&\frac{1}{T}\int_0^T \sum_{\gamma\neq \gamma'} \rho(\gamma-t)\rho(\gamma'-t) e\big(\alpha_1\tfrac{\log T}{2\pi} (\gamma-t) + \alpha_2\tfrac{\log T}{2\pi}(\gamma'-t)\big)\,dt\notag\\
&\; = o(1) + (1+o(1))\int_{\mathbb{R}^2} \rho\big(\tfrac{2\pi\nu_1}{\log T}\big)\rho\big(\tfrac{2\pi\nu_2}{\log T}\big) e(\alpha_1\nu_1 + \alpha_2\nu_2)\bigg[1-\Big(\frac{\sin \pi (\nu_1-\nu_2)}{\pi (\nu_1-\nu_2)}\Big)^2\bigg]\,d\nu_1d\nu_2.
\end{align}
uniformly for $\alpha_1$ and $\alpha_2$ in a fixed bounded region. This can be shown by modifying the previous argument, although we leave the (somewhat tedious and secondary) details to the reader.

Said somewhat more informally: because we may integrate in $\alpha_1$ and $\alpha_2$, the measures
\begin{equation}
\label{zetameas}
\bigg[\frac{1}{T}\int_0^T\sum_{\gamma\neq\gamma'} \delta\big(\nu_1-\tfrac{\log T}{2\pi}(\gamma-t)\big)\delta\big(\nu_2-\tfrac{\log T}{2\pi}(\gamma'-t)\big)\bigg]\rho\big(\tfrac{2\pi\nu_1}{\log T}\big)\rho\big(\tfrac{2\pi\nu_2}{\log T}\big)\,d\nu_1d\nu_2,
\end{equation}
and
\begin{equation}
\label{GUEmeas}
\bigg[1-\Big(\frac{\sin \pi (\nu_1-\nu_2)}{\pi (\nu_1-\nu_2)}\Big)^2\bigg]\rho\big(\tfrac{2\pi\nu_1}{\log T}\big)\rho\big(\tfrac{2\pi\nu_2}{\log T}\big)\,d\nu_1d\nu_2,
\end{equation}
are asymptotically indistinguishable with respect to test functions that have a fixed compact Fourier support and therefore do not concentrate themselves too narrowly. This latter measure, of course, is the limiting pair correlation measure associated to the Gaussian Unitary Ensemble (GUE).

There is nothing special about our use of the function $\rho$ here. Its placement serves only to cutoff measures \eqref{zetameas} and \eqref{GUEmeas} so they are (close to being) supported in a square region with dimensions of order $\log T$. This is an important feature of Montgomery's conjecture -- outside of this region the random-matrix-theory measure given by \eqref{GUEmeas} ceases to be as good an approximation to \eqref{zetameas}.

Even inside this region -- to which we will restrict our attention in this paper -- the measure \eqref{GUEmeas} misses important phenomena that will be important if we are to achieve a stronger estimate than \eqref{SPC2}. These are the troughs near low-lying zeroes in histograms of $\gamma-\gamma'$, seen in Figure \ref{data}.

Because the 1-level density of zeroes near $T$ is not stationary but grows like $\tfrac{\log(T/2\pi)}{2\pi}$, measure \eqref{zetameas} will even more closely resemble measure \eqref{GUEmeas} if the average from $0$ to $T$ is replaced by an average from $T$ to $T+H$ for $H = o(T)$. We will do so in the sequel, and in addition, for technical reasons, we will use smoothed averages; we replace
$$
\frac{1}{H}\int_T^{T+H}\cdot\cdot\cdot\,dt = \frac{1}{H}\int_\mathbb{R} \mathbf{1}_{[0,1]}\Big(\tfrac{t-T}{H}\Big)\cdot\cdot\cdot \,dt
$$
by
$$
\frac{1}{H}\int_\mathbb{R}\sigma\Big(\tfrac{t-T}{H}\Big)\cdot\cdot\cdot\,dt,
$$
where $\sigma$ is some smooth bump function with mass $1$. Making use of such smoothed averages makes the computations that follows easier than they would otherwise be.

We prove
\begin{thm}
\label{main}
For fixed $\sigma$ and $h$ with $\hat{\sigma}, \hat{h}$ smooth and compactly supported, and $\sigma$ of mass 1; fixed $\epsilon, \kappa > 0$; $L$ within a bounded distance $\kappa$ of $\log T$; and $H \leq T$,
\begin{align}
\label{mainn}
&\frac{1}{H} \int_\mathbb{R} \sigma\Big(\tfrac{t-T}{H}\Big)\sum_{\gamma\neq \gamma'} h(\gamma-t,\gamma'-t) e\big(\alpha_1\tfrac{L}{2\pi}(\gamma-t)+\alpha_2 \tfrac{L}{2\pi}(\gamma'-t)\big) \,dt \notag\\
&\;= O\Big(\frac{T^{|\tfrac{\alpha_1}{2}|+|\tfrac{\alpha_2}{2}|}}{H}\Big) + O\Big(\log T\Big(\frac{H}{T}+\frac{1}{H}\Big)\Big)  \\
&\;\;+ \int_{-\infty}^\infty \int_{-\infty}^\infty \bigg[\Big(\frac{\log(T/2\pi)}{2\pi}\Big)^2 + Q_T(\nu_1-\nu_2)\bigg] h(\nu_1,\nu_2)e\big(\alpha_1\tfrac{L}{2\pi}\nu_1+\alpha_2 \tfrac{L}{2\pi}\nu_2\big) \,d\nu_1 d\nu_2 \notag
\end{align}
uniformly for $(|\alpha_1|+ |\alpha_2|)/2 \leq 1-\epsilon.$

Here
\begin{align*}
Q_t(u) := &\frac{1}{4\pi^2}\bigg(\Big(\frac{\zeta'}{\zeta}\Big)'(1+iu) - B(iu) + \Big(\frac{\zeta'}{\zeta}\Big)'(1-iu) - B(-iu) \\&\;+ \big(\frac{t}{2\pi}\big)^{-iu}\zeta(1-iu)\zeta(1+iu)A(iu) + \big(\frac{t}{2\pi}\big)^{iu}\zeta(1+iu)\zeta(1-iu)A(-iu)\bigg),
\end{align*}
defined by continuity at $u=0$, and
$$
A(s) := \prod_p \frac{(1-\tfrac{1}{p^{1+s}})(1-\tfrac{2}{p}+\tfrac{1}{p^{1+s}})}{\big(1-\tfrac{1}{p}\big)^2} = \prod_p \bigg(1-\frac{(1-p^{-s})^2}{(p-1)^2}\bigg) = 1+O(s^2),
$$
and
$$
B(s) := \sum_p \frac{\log ^2 p}{(p^{1+s}-1)^2}.
$$
\end{thm}

\textit{Remark:} One may, of course, optimize by setting $H = T^{1-\epsilon/2}/\sqrt{\log T}$. The reason we have written our error terms in this manner is that the second error term $O(\log T(H/T+1/H))$ is somewhat artificial. This will become clear in the proof; at the cost of a somewhat more baroque result it could be eliminated. The remaining error term is somewhat more fundamental -- and at any rate both effectively amount to an inverse power-of-$T$ error term.

\textit{Remark:} Here and in what follows we adopt the convention of counting zeroes with multiplicity, so that in particular for a function $f$, a sum
$$
\sum_{\gamma\neq \gamma'} f(\gamma,\gamma')
$$
is really
$$
\sum_{\gamma, \gamma'} f(\gamma,\gamma') - \sum_{\gamma} f(\gamma,\gamma)
$$
In all likelihood every zero of the zeta function occurs with multiplicity $1$, but provided we adopt this notational convention we do not need to assume so.

\textit{Remark:} This theorem is of interest mainly in the case that $\alpha_1 = -\alpha_2$, as in Conjecture \ref{SPC}. We do not specialize to this case only because in the computation of higher order correlation functions such a specialization ceases to be as natural.

A pair correlation function of this form was first conjectured by Bogomolny and Keating in \cite{BogKeat}, on part on an analogy from the field of quantum chaos. Further support for this form was offered by Conrey and Snaith \cite{ConSnaith}, who showed that it can be derived from the ratio conjecture of Conrey, Farmer, and Zirnbauer \cite{CFZ}. It bears remarking that, in the semiclassical language, we shall only really see the diagonal terms $(\zeta'/\zeta)'(1+iu) - B(iu)$ and conjugate because of the restrictions we place on $\alpha_1$ and $\alpha_2$. Indeed, apart from some analytic devices which mimic the large sieve, we recover these terms in much the same fashion as Bogomolny and Keating. There are no new arithmetical ideas required; to rigorously extend the range in which $\alpha_1, \alpha_2$ lie and effectively see the off-diagonal terms that remain is more difficult and will almost certainly require a breakthrough.

By re-averaging Theorem \ref{main} in $T$, and using a 1-level density estimate as before, we show that
\begin{thm}
\label{4corr}
For fixed $\epsilon > 0$ and fixed $\omega$ with a smooth and compactly supported Fourier transform,
\begin{align}
\label{4correq}
&\frac{1}{T} \sum_{0 < \gamma \neq \gamma' \leq T} \omega(\gamma-\gamma')e\big(\alpha \tfrac{\log T}{2\pi}(\gamma-\gamma')\big)\\
&\;=O_\delta\Big(\frac{1}{T^\delta}\Big)+\int_\mathbb{R} \omega(u)e\big(\alpha \tfrac{\log T}{2\pi}u\big)\bigg[\frac{1}{T}\int_0^T \Big(\frac{\log(t/2\pi)}{2\pi}\Big)^2 + Q_t(u)\,dt\bigg]\,du \notag
\end{align}
for any $\delta < \epsilon/2$, uniformly for $|\alpha|< 1-\epsilon$.
\end{thm}
In fact, by proceeding carefully in the analysis that follows, one can show this even for Montgomery's choice of test function $\omega(u) = w(u)$, owing to the rapid decrease of this function's Fourier transform. We leave this refined computation to the reader. That the integral on the right extends over all of $\mathbb{R}$, rather than simply $[-T,T]$ may seem surprising since from the left hand side we inherit only information about $\omega$ in the latter interval, but the decay of $\omega$ is in all cases sufficient that the difference between these two regions of integration is absorbed into the error term.

The sum over zeroes in \eqref{4correq} is of the same form as that in the Strong Pair Correlation Conjecture's \eqref{strongpair}. (Should we push our analysis to include $\omega = w$, they would be the same sum.) Our result therefore refines Montgomery's original work. Indeed, one can show that the difference between this prediction and a somewhat naive extension of the GUE prediction in which
$$
Q_t(u)
$$
has been replaced by
$$
K_t(u) = -\Big(\frac{\log(t/2\pi)}{2\pi}\Big)^2 K\Big(\frac{\log(t/2\pi)}{2\pi}u\Big),
$$
where $K(u) = \big(\tfrac{\sin \pi u}{\pi u}\big)^2$, is larger than the error term. We conclude our paper with a demonstration of the difference between these two predictions.

For $\alpha$ not restricted so stringently: there is no reason not to believe that for any fixed compact region $M$, Theorem \ref{4corr} is true for all $\alpha\in M$ for any $\delta < 1/2$.

\textit{Remark:} By integrating in $\alpha$ in Theorem \ref{4corr}, one can see the statistics
$$
\sum_{0 < \gamma \neq \gamma' \leq T} f\big(\tfrac{\log T}{2\pi}(\gamma-\gamma'-s)\big) \omega(\gamma-\gamma'),
$$
to a uniform error of $O(\|\hat{f}\|_{L^1}/T^\delta)$ for any $s$, where $f$ is supported in $[-1+\epsilon,1-\epsilon]$ and $\delta < \epsilon /2$. This is the precise sense in which one can see the troughs in the pair correlation measure. If one wants these statistics only for fixed $s$, without uniformity, the computations that follow can be simplified slightly.

\textit{Remark:} The above statistics say nothing about the case that our macroscopic interval grows. For instance, they say nothing about the sums
\begin{equation}
\label{expandingmacro}
\frac{1}{T}\sum_{0 < \gamma \neq \gamma' \leq T} \omega\Big(\frac{\gamma-\gamma'}{T^\nu}\Big)e\big(\alpha \tfrac{\log T}{2\pi}(\gamma-\gamma')\big)
\end{equation}
when $0 < \nu < 1$. The statistics over such intervals would be expected to deviate from GUE statistics in the same way as above, and bear relevance to the variance statistics first obtained by Fujii \cite{Fu} for the number of zeroes lying in an interval $[t,t+h]$, where $h\rightarrow\infty$ as $t\rightarrow\infty$.

In fact, using the method below, one can evaluate the sums \eqref{expandingmacro} to a high degree of accuracy, at the cost of imposing additional band-limits on the Fourier variable $\alpha$, beyond those of Theorems \ref{main} and \ref{4corr}. However, we do not pursue the matter here.

\section{The explicit formula and an outline of a proof}

Our main tool in establishing this will be the well known explicit formula relating the zeroes of the Zeta function to the primes. We define
$$
\Omega(\xi) := \tfrac{1}{2}\tfrac{\Gamma'}{\Gamma}\big(\tfrac{1}{4}+i\tfrac{\xi}{2}\big) +\tfrac{1}{2}\tfrac{\Gamma'}{\Gamma}\big(\tfrac{1}{4}-i\tfrac{\xi}{2}\big) - \log \pi.
$$

\begin{thm}
[The explicit formula]
Let $g$ a measurable function such that $g(x) = \tfrac{g(x+)+g(x-)}{2}$, and for some $\delta > 0$,
$$
(a)\quad \int_{-\infty}^\infty e^{(\tfrac{1}{2}+\delta)|x|}|g(x)|dx < +\infty,
$$
$$
(b)\quad \int_{-\infty}^\infty e^{(\tfrac{1}{2}+\delta)|x|}|dg(x)| < +\infty.
$$
Then we have
$$
\lim_{L\rightarrow\infty}\sum_{|\gamma|<L}\hat{g}\big(\tfrac{\gamma}{2\pi}\big)- \int_{-L}^L \frac{\Omega(\xi)}{2\pi}\hat{g}\big(\tfrac{\xi}{2\pi}\big)\,d\xi = \int_{-\infty}^\infty [g(x)+g(-x)]e^{-x/2}\,d\big(e^x-\psi(e^x)\big),
$$
where here $\psi(x) = \sum_{n\leq x} \Lambda(n)$, for the von Mangoldt function $\Lambda$.
\end{thm}

The explicit formula in this generality is due to Weil, \cite{Weil} but before him something very much like it was written down in varying degrees by Riemann \cite{Rie}, and Guinand \cite{Gui}. Note that if
$$
\tilde{d}(\xi) = \sum_\gamma \delta_\gamma(\xi),
$$
the left hand side of the explicit formula becomes the principal value integral
$$
-\!\!\!\!\!\!\!\int_{-\infty}^\infty \hat{g}\big(\tfrac{\xi}{2\pi}\big) \big[\tilde{d}(\xi)-\tfrac{\Omega(\xi)}{2\pi}\big]d\xi.
$$
(In what follows we will be working with nice enough functions that it will not matter that this is a principal value integral.) If we define $S(T)$ in the standard way (see \cite{MontVau} pp. 452) so that
$$
N(T) = \frac{1}{\pi}\arg \Gamma\big(\tfrac{1}{4}+i\tfrac{T}{2}\big)-\tfrac{T}{2\pi}\log \pi + 1 + S(T),
$$
then (only on the Riemann Hypothesis)
$$
\big[\tilde{d}(\xi)-\tfrac{\Omega(\xi)}{2\pi}\big]d\xi = dS(\xi)
$$
This is, if nothing else, a notational convenience, and we will make use of it for that reason in what follows.

$S(T)$ is relatively small and oscillatory, so that $\tfrac{\Omega(\xi)}{2\pi}$ is an expression for the mean density of zeroes around $\xi$. By Stirling's formula,
$$
\frac{\Omega(\xi)}{2\pi} = \frac{\log\big((|\xi|+2)/2\pi\big)}{2\pi} + O\Big(\frac{1}{|\xi|+2}\Big).
$$
It is therefore clear in this formulation that the explicit formula expresses a Fourier duality between the error term in the prime number theorem and the error term of the zero-counting function.

The explicit formula is proven by a simple contour integration argument, making use of the the reflection formula to evaluate one-half of the contour. (For the outline of a proof, with the final result stated slightly differently, see \cite{Iwa} pp. 108 or \cite{MontVau} pp. 410.)

We will also need another result which makes an appearance in \cite{MontVau} -- in fact it is what is used by the authors to state the explicit formula differently. This is

\begin{lem}
\label{Digamma}
Let $a > 0$ and $b > 0$ be fixed. If $J \in L^1(\mathbb{R})$, $J$ is of bounded variation on $\mathbb{R}$, and if $J(x) = J(0) + O(|x|)$, then
$$
\lim_{T\rightarrow\infty}\int_{-T}^T \tfrac{\Gamma'}{\Gamma}(a\pm i bt)\hat{J}(t) dt = \tfrac{\Gamma'}{\Gamma}(a)J(0) + \int_0^\infty \frac{e^{-ay}}{1-e^{-y}}\Big[J(0)-J\Big(\mp\tfrac{by}{2\pi}\Big)\Big] dx.
$$
\end{lem}

\begin{proof}
See \cite{MontVau} pp. 414.
\end{proof}

\textit{An outline of a proof of Theorem \ref{main}:} This outline will be expanded upon in section 4. We prove our result with a series of computational Lemmas. To simplify exposition, we will deal only with the case that $h(\nu_1,\nu_2) = r(\nu_1/2\pi)r(\nu_2/2\pi)$ for $\hat{r}$ smooth and compactly supported. The proof in general is identical, but writing $h$ as a product of two functions draws some parts of the proof into clearer light. We dilate the functions $r$ for the sake of tidying up some of the formulas that follow.

Because $\big[\tilde{d}(\xi)-\tfrac{\Omega(\xi)}{2\pi}\big]d\xi = dS(\xi)$, and because we do not expect that $\Omega(\xi)/2\pi$ to correlate at a local scale with $\tilde{d}(\xi)$, to understand the pair correlation statistics $\tilde{d}(\xi_1+t)\tilde{d}(\xi_2+t)$ it is enough to understand those of $dS(\xi_1+t)dS(\xi_2+t)$. In fact, in many ways, the statistics of the latter are a more natural quantity to consider. The second error term in Theorem \ref{main} alluded to earlier is due to the passage back to the former.

It will follow from computation with the explicit formula that the smoothed average
$$
\Bigg\langle \int\int_{\mathbb{R}^2} r\Big(\frac{\xi_1}{2\pi}\Big) r\Big(\frac{\xi_2}{2\pi}\Big)e\big(\alpha_1\tfrac{L}{2\pi}\xi_1 + \alpha_2 \tfrac{L}{2\pi}\xi_2\big) dS(\xi_1+t)dS(\xi_2+t)\Bigg\rangle_{t\in[T,T+H]}
$$
is very close to
\begin{align*}
&\sum_{\log n - \log m = O(1/T)} \big(\hat{r}(-\log n - \alpha_1 L) \hat{r}(\log m - \alpha_2 L ) + \hat{r}(\log n - \alpha_1 L)\hat{r}(- \log m - \alpha_2 L)\big)\frac{\Lambda^2(n)}{n}\\
&\;=\sum_{n=1}^\infty \big(\hat{r}(-\log n - \alpha_1 L) \hat{r}(\log n - \alpha_2 L ) + \hat{r}(\log n - \alpha_1 L)\hat{r}(- \log n - \alpha_2 L)\big)\frac{\Lambda^2(n)}{n}
\end{align*}
as in passing from the first line both $n$ and $m$ must be less than $T$, for $\alpha_1, \alpha_2 \leq 1$ by the compact support of $\hat{r}$. (Recall $L$ is near $\log T$.) These are the terms inherited from the measure $d\psi(e^x)$ in the explicit formula; those terms which are inherited from $d(e^x)$ drop out, roughly speaking, because they have no discrete part. These ideas date back to Montgomery's original proof in \cite{Montgomery}. We are able to obtain slightly better error terms only by using smoothed averages. This suppresses the appearance of any large-sieve-type inequalities. This is the content of Lemma \ref{6}.

One can already see the diagonal terms in this, but care is needed to do so rigorously. This is done in Lemma \ref{7}.

We do not recover the off-diagonal terms in this way, only a proxy for them, but we show in Lemma \ref{offdiagonal} that for the restricted class of test functions we consider there is no difference between the two.

Finally, in Lemmas \ref{1level} and \ref{11} we show that our intuition about being able to recover $\tilde{d}(\xi_1+t)\tilde{d}(\xi_2+t)$ from $dS(\xi_1+t)dS(\xi_2+t)$ is correct.

\section{Lemmata}

\begin{lem}
\label{6}
For fixed $r$ and $\sigma$, with $\hat{r}$ and $\hat{\sigma}$ smooth and compactly supported and $\sigma$ of mass 1, we have
\begin{align}
\label{a}
&\frac{1}{H}\int_\mathbb{R}\sigma\Big(\tfrac{t-T}{H}\Big)\int_{-\infty}^\infty\int_{-\infty}^\infty r\Big(\tfrac{\xi_1-t}{2\pi}\Big) r\Big(\tfrac{\xi_2-t}{2\pi}\Big) e\big(\alpha_1 \tfrac{L}{2\pi} (\xi_1-t) + \alpha_2 \tfrac{L}{2\pi}(\xi_2-t)\big)\, dS(\xi_1)dS(\xi_2)\,dt\notag\\
&=  \sum_{n=1}^\infty \big[\hat{r}(-\log n-\alpha_1 L)\hat{r}(\log n-\alpha_2 L)+\hat{r}(\log n-\alpha_1 L)\hat{r}(-\log n-\alpha_2 L)\big]\frac{\Lambda^2(n)}{n}\\
&\;\;\;+O\bigg(\frac{e^{(|\tfrac{\alpha_1}{2}| + |\tfrac{\alpha_2}{2}|)L}}{H}\bigg).\notag
\end{align}
Here the implied constant depends upon $r$ and $\sigma.$
\end{lem}

\begin{proof} Note that $\hat{g}(\tfrac{\xi}{2\pi}) = r\big(\tfrac{\xi-t}{2\pi}\big) e\big(\alpha \tfrac{L}{2\pi} (\xi-t)\big)$ when $g(x) = e\big(\tfrac{xt}{2\pi}\big)\hat{r}(-(x+\alpha L)).$ As this function has compact support, we are justified in using the explicit formula to evaluate the left hand side of \eqref{a}. It is equal to
\begin{align}
\label{b}
\frac{1}{H}\int_\mathbb{R}\sigma\Big(\tfrac{t-T}{H}\Big)&\sum_{\varepsilon\in\{-1,1\}^2}\int_{-\infty}^\infty\int_{-\infty}^\infty e(-\tfrac{t}{2\pi}(\varepsilon_1x_1+\varepsilon_2x_2))
\hat{r}(\varepsilon_1x_1-\alpha_1L) \notag\\&\cdot\hat{r}(\varepsilon_2x_2-\alpha_2L)
e^{-(x_1+x_2)/2} d\big(e^{x_1}-\psi(e^{x_1})\big)d\big(e^{x_2}-\psi(e^{x_2})\big)\notag\\
= \sum_{\varepsilon\in\{-1,1\}^2}\int_{-\infty}^\infty&\int_{-\infty}^\infty e(-\tfrac{T}{2\pi}(\varepsilon_1x_1+\varepsilon_2x_2))\hat{\sigma}\big(\tfrac{H}{2\pi}(\varepsilon_1x_1+\varepsilon_2x_2)\big)
\hat{r}(\varepsilon_1x_1-\alpha_1L) \\
&\cdot\hat{r}(\varepsilon_2x_2-\alpha_2L)
e^{-(x_1+x_2)/2} d\big(e^{x_1}-\psi(e^{x_1})\big)d\big(e^{x_2}-\psi(e^{x_2})\big),\notag
\end{align}
where the interchange of integrals is justified by Fubini's theorem.

Note that, for instance,
\begin{align*}
&\int_{-\infty}^\infty  \hat{\sigma}\big(\tfrac{H}{2\pi}(\varepsilon_1x_1 + \varepsilon_2x_2)\big)\hat{r}(\varepsilon_1x_1-\alpha_1L) \hat{r}(\varepsilon_2x_2-\alpha_2L)e^{-(x_1+x_2)/2}d(e^{x_1})\\
&\lesssim_{r,\sigma}\hat{r}(\varepsilon_2x_2-\alpha_2L)e^{-x_2/2}\frac{e^{\min(|\alpha_1|L,x_2)/2}}{H},
\end{align*}
and if $M(x)$ is either of the functions $x$ or $\psi(x)$,
$$
\int_{-\infty}^\infty \hat{r}\big((\varepsilon_2x_2-\alpha_2L)\big) e^{-x_2/2} e^{|\alpha_1|L/2}dM(e^{x_2}) \lesssim_r e^{(\tfrac{|\alpha_1|+|\alpha_2|}{2})L}.
$$
In this way, expanding the measure
\begin{align*}
&d\big(e^{x_1}-\psi(e^{x_1})\big)d\big(e^{x_2}-\psi(e^{x_2})\big)\\&\;= d(e^{x_1})d(e^{x_2}) - d(x^{x_2})d\psi(e^{x_1}) - d(e^{x_1})d\psi(e^{x_2})+d\psi(x^{x_1})d\psi(e^{x_2}),
\end{align*}
the integrand in \eqref{a} integrated with respect to all terms but $d\psi(e^{x_1})d\psi(e^{x_2})$ is bound by
$$
O\bigg(\frac{e^{(\tfrac{\alpha_1}{2} + \tfrac{\alpha_2}{2})L}}{H}\bigg),
$$
for any $\varepsilon\in\{-1,1\}^2.$

On the other hand,
\begin{align*}
\sum_{\varepsilon\in\{-1,1\}^2}\int_{-\infty}^\infty&\int_{-\infty}^\infty e(-\tfrac{T}{2\pi}(\varepsilon_1x_1+\varepsilon_2x_2))\hat{\sigma}\big(\tfrac{H}{2\pi}(\varepsilon_1x_1+\varepsilon_2x_2)\big)
\hat{r}(\varepsilon_1x_1-\alpha_1L) \notag\\&\cdot\hat{r}(\varepsilon_2x_2-\alpha_2L)
e^{-(x_1+x_2)/2} d\psi(e^{x_1})d\psi(e^{x_2})\\
=\sum_{\varepsilon\in\{-1,1\}^2}&\sum_{n_1,n_2=1}^\infty e(-\tfrac{T}{2\pi}(\varepsilon_1\log n_1+\varepsilon_2\log n_2))\hat{\sigma}\big(\tfrac{H}{2\pi}(\varepsilon_1\log n_1+\varepsilon_2\log n_2)\big)
\hat{r}(\varepsilon_1\log n_1-\alpha_1L) \notag\\&\cdot\hat{r}(\varepsilon_2\log n_2-\alpha_2L) \frac{\Lambda(n_1)\Lambda(n_2)}{\sqrt{n_1n_2}}
\end{align*}
We will consider each $\varepsilon \in \{-1,1\}^2$ in turn. That $\hat{\sigma}$ is compactly supported implies that we may restrict our sum to those $(n_1,n_2)$ with $\varepsilon_1 \log n_1 + \varepsilon_2 \log n_2 = O\big(\tfrac{1}{H}\big)$. But that $\hat{r}$ is compactly supported implies that we may restrict our sum to those $n_1 = \exp[\varepsilon_1\alpha_1L + O(1)], n_2 = \exp[\varepsilon_2\alpha_2L + O(1)]$. Since $n_1, n_2$ are positive integers, for sufficiently large $H$ the terms with $\varepsilon_1 = \varepsilon_2 = 1$ or $\varepsilon_1 = \varepsilon_2 = -1$ will be null. On the other hand, for $\varepsilon_1 =1, \varepsilon_2 = -1$ or $\varepsilon_1 = -1, \varepsilon_2 = 1$, if $n_1 \neq n_2$, we have $|n_1-n_2| \geq 1$, so
$$
|\varepsilon_1 \log n_1 + \varepsilon_2 \log n_2| \gtrsim \frac{1}{\sqrt{n_1 n_2}},
$$
and so if $H / e^{(|\tfrac{\alpha_1}{2}|+|\tfrac{\alpha_2}{2}|)L}$ is sufficiently large, the only $n_1, n_2$ for which $\varepsilon_1 \log n_1 + \varepsilon_2 \log n_2 = O\big(\tfrac{1}{H}\big)$ are those for which $n_1 = n_2$ so that $\varepsilon_1 \log n_1 + \varepsilon_2 \log n_2 = 0$. This gives \eqref{a}, as the left hand side is always $O(1)$, to cover the case that $e^{(|\tfrac{\alpha_1}{2}|+|\tfrac{\alpha_2}{2}|)L}/H$ is large.
\end{proof}

It is easy to formally discern the term
$$
\Big(\frac{\zeta'}{\zeta}\Big)'(1+iu)-B(iu) = \sum \frac{\log n \,\Lambda(n)}{n^{1+iu}}-\sum_p \frac{\log^2 p}{(p^{1+iu}-1)^2} = \sum \frac{\Lambda^2(n)}{n^{1+iu}}
$$
here, but some care is needed to deal with issues of convergence, near $u = 0$ especially. One can either make use of the explicit formula, or in some manner reprove it, and we follow the first path.

\begin{lem}
\label{7}
For $\min(|\alpha_1|,|\alpha_2|)\leq 1-\epsilon$ and $r$ with $\hat{r}$ smooth and compactly supported, and for $\lambda - (1-\epsilon)L$ sufficiently large (where `sufficiently large' depends only on $\epsilon$ and the region in which $\hat{r}$ can be supported),
\begin{align}
\label{c}
&\sum_{n=1}^\infty \big[\hat{r}(-\log n - \alpha_1 L)\hat{r}(\log n  - \alpha_2 L) + \hat{r}(\log n - \alpha_1 L) \hat{r}(-\log n - \alpha_2 L)\big] \frac{\log(n)\Lambda(n)}{n}\notag\\
&\;=\int_{-\infty}^\infty\int_{-\infty}^\infty r(\nu_1)r(\nu_2)e(\alpha_1 L \nu_1 + \alpha_2 L \nu_2)\bigg[\lambda \delta(\nu_1-\nu_2) + \Big(\frac{\zeta'}{\zeta}\Big)'\big(1+i2\pi(\nu_1-\nu_2)\big)\\
&\;\;\;\;\;\;\;\;\;\;\;\;\;\;\;\;\;\;\;\;\;\;\; + \Big(\frac{\zeta'}{\zeta}\Big)'\big(1-i2\pi(\nu_1-\nu_2)\big) + \frac{e((\nu_1-\nu_2)\lambda)+ e(-(\nu_1-\nu_2)\lambda)}{(2\pi (\nu_1-\nu_2))^2}\bigg]d\nu_1d\nu_2\notag
\end{align}
\end{lem}

\textit{Remark: } It may seem odd that we have such freedom in choosing $\lambda$ in this Lemma. We are able to do so because our restrictions on $\alpha_1$ and $\alpha_2$ mean that we will not see a large part of the measure against which we integrate $r(\nu_1)r(\nu_2)e(\alpha_1 L \nu_1 + \alpha_2 L \nu_2)$. We will eventually want $\lambda = \log(T/2\pi)$ and $L = \log(T) + O(1)$, and it is worthwhile to keep this in mind.

\begin{proof}
The left hand side of \eqref{c} is
$$
\int_{-\infty}^\infty \big[\hat{r}(-x-\alpha_1 L)\hat{r}(x-\alpha_2 L) \tfrac{|x|}{e^{|x|/2}}+\hat{r}(x-\alpha_1 L)\hat{r}(-x-\alpha_2 L) \tfrac{|x|}{e^{|x|/2}}\big] e^{-x/2} d\psi(e^x).
$$
For $q(x) = \hat{r}(-x-\alpha_1 L)\hat{r}(x-\alpha_2 L) \tfrac{|x|}{e^{|x|/2}}$,
\begin{align*}
\hat{q}(\xi) =& \int_{-\infty}^\infty \int_{-\infty}^\infty r(\nu_1)r(\nu_2) e(\alpha_1 L \nu_1 + \alpha_2 L \nu_2)\\
&\;\;\;\;\;\;\;\;\;\cdot\bigg(\frac{1}{\big(\tfrac{1}{2}+i2\pi(\nu_1-\nu_2-\xi)\big)^2}+\frac{1}{\big(\tfrac{1}{2}-i2\pi(\nu_1-\nu_2-\xi)\big)^2}\bigg) \,d\nu_1d\nu_2
\end{align*}
so that by the explict formula the left hand side of \eqref{c} is
\begin{align*}
&\int_{-\infty}^\infty \big[\hat{r}(-x-\alpha_1 L)\hat{r}(x-\alpha_2 L) +\hat{r}(x-\alpha_1 L)\hat{r}(-x-\alpha_2 L) \big] |x| e^{-(\tfrac{x}{2}+\tfrac{|x|}{2})} d(e^x)\\ &\;\;\;\;\;\;\;\;- \int_{-\infty}^\infty \int_{-\infty}^\infty \int_{-\infty}^\infty r(\nu_1)r(\nu_2) e(\alpha_1L\nu_1 + \alpha_2L\nu_2)\\
&\;\;\;\;\;\;\;\;\;\;\;\;\;\cdot\bigg(\frac{1}{\big(\tfrac{1}{2}+i2\pi(\nu_1-\nu_2-\xi)\big)^2}+\frac{1}{\big(\tfrac{1}{2}-i2\pi(\nu_1-\nu_2-\xi)\big)^2} \bigg)d\nu_1d\nu_2 \,dS(\xi)
\end{align*}
The first of these terms is
$$
\int_{-\infty}^\infty |x| \hat{r}(-x-\alpha_1 L)\hat{r}(x-\alpha_2 L) \,dx + \int_{-\infty}^\infty |x| e^{-|x|}\hat{r}(-x-\alpha_1 L)\hat{r}(x-\alpha_2 L) \,dx,
$$
and because $\hat{r}$ is compactly supported and one of $|\alpha_1|$, $|\alpha_2|$ is no greater than $1-\epsilon$, for sufficiently large $(\lambda - (1-\epsilon)L)$ the first integral is equal to
\begin{align*}
&\int_{-\infty}^\infty \lambda\cdot\Big[1-\big(1-|\tfrac{x}{\lambda}|\big)_+\Big]\hat{r}(-x-\alpha_1 L)\hat{r}(x-\alpha_2 L)\,dx\\
& =  \int_{-\infty}^\infty \int_{-\infty}^\infty r(\nu_1)r(\nu_2)e(\alpha_1 L \nu_1 + \alpha_2 L \nu_2)\big[\lambda \delta(\nu_1-\nu_2)-\lambda^2\Big(\frac{\sin \pi \lambda(\nu_1-\nu_2)}{\pi \lambda(\nu_1-\nu_2)}\Big)^2\big]\,d\nu_1d\nu_2.
\end{align*}
It is worth remarking that we can make this transition only because of the restricted range of $\alpha_1, \alpha_2$.

Evaluation of the second integral is routine, and we have that the left hand side of \eqref{c} is
\begin{align}
\label{d}
&\int_{-\infty}^\infty \int_{-\infty}^\infty r(\nu_1)r(\nu_2)e(\alpha_1 L \nu_1 + \alpha_2 L \nu_2) \notag\\
&\;\;\;\;\;\;\;\;\;\cdot \Bigg[\lambda\delta(\nu_1-\nu_2) - \lambda^2\Big(\frac{\sin \pi (\nu_1-\nu_2)}{\pi (\nu_1-\nu_2)}\Big)^2+ \bigg(\frac{1}{\big(1+i2\pi(\nu_1-\nu_2)\big)^2}+\frac{1}{\big(1-i2\pi(\nu_1-\nu_2)\big)^2}\bigg)\\
&\;\;\;\;\;\;\;\;\;\;\;\;-\int_{-\infty}^\infty  \bigg(\frac{1}{\big(\tfrac{1}{2}+i2\pi(\nu_1-\nu_2-\xi)\big)^2} +\frac{1}{\big(\tfrac{1}{2}-i2\pi(\nu_1-\nu_2-\xi)\big)^2}\bigg)dS(\xi)\Bigg]d\nu_1d\nu_2.\notag
\end{align}
Note that for $s$ real, by Lemma \ref{Digamma} (or contour integration),
$$
\int_{-\infty}^\infty \bigg(\frac{1}{\big(\tfrac{1}{2}+i(s-\xi)\big)^2}+ \frac{1}{\big(\tfrac{1}{2}-i(s-\xi)\big)^2}\bigg)\frac{\Omega(\xi)}{2\pi}\,d\xi = -\frac{1}{4}\Big((\tfrac{\Gamma'}{\Gamma})'\big(\tfrac{1}{2}(1+is)\big) + (\tfrac{\Gamma'}{\Gamma})'\big(\tfrac{1}{2}(1-is)\big)\Big)
$$
Recall that for any $s$, (differentiating the classical representation 2.12.7 in \cite{Titch}),
\begin{equation}
\label{zetadiff}
\Big(\frac{\zeta'}{\zeta}\Big)'(s) = \frac{1}{(s-1)^2} + \frac{1}{s^2} - \frac{1}{4}(\tfrac{\Gamma'}{\Gamma})'\big(\tfrac{1}{2}s\big) -\sum_\gamma \frac{1}{\big(s-(\tfrac{1}{2}+i\gamma)\big)^2},
\end{equation}
On expanding $\big(\tfrac{\sin u}{u}\big)^2 = \tfrac{2}{(2u)^2}-\tfrac{e^{i2u}+e^{-i2u}}{(2u)^2}$, we see the expression \eqref{d} is just
\begin{align*}
&\int_{-\infty}^\infty\int_{-\infty}^\infty r(\nu_1)r(\nu_2)e(\alpha_1 L \nu_1 + \alpha_2 L \nu_2)\bigg[\lambda \delta(\nu_1-\nu_2) + \Big(\frac{\zeta'}{\zeta}\Big)'\big(1+i2\pi(\nu_1-\nu_2)\big)\\
&\;\;\;\;\;\;\;\;\;\;\;\;\;\;\;\;\;\;\;\;\;\;\; + \Big(\frac{\zeta'}{\zeta}\Big)'\big(1-i2\pi(\nu_1-\nu_2)\big) + \frac{e((\nu_1-\nu_2)\lambda)+ e(-(\nu_1-\nu_2)\lambda)}{(2\pi (\nu_1-\nu_2))^2}\bigg]d\nu_1d\nu_2
\end{align*}
\end{proof}

We need that the prediction in \eqref{c} does not differ so much from that in Theorem \ref{main}:

\begin{lem}
\label{offdiagonal}
For $h$ with $\hat{h}$ smooth and compactly supported, so long as $\lambda \geq (\tfrac{|\alpha_1|}{2}+\tfrac{|\alpha_2|}{2})L$,
\begin{align*}
&\int_{-\infty}^\infty\int_{-\infty}^\infty h(\nu_1,\nu_2)e(\alpha_1 L \nu_1 + \alpha_2 L \nu_2) \bigg[-\frac{e(-(\nu_1-\nu_2)\lambda)}{(2\pi(\nu_1-\nu_2))^2}\\
&\;\;\;\;\;+e(-(\nu_1-\nu_2)\lambda)\zeta(1-i2\pi(\nu_1-\nu_2)) \zeta(1+i2\pi(\nu_1-\nu_2))A(i2\pi(\nu_1-\nu_2))\bigg]d\nu_1d\nu_2\\
&\;\;\;\;= O_h\bigg(\frac{1}{e^{\lambda-(\tfrac{|\alpha_1|}{2}+\tfrac{|\alpha_2|}{2})L}}\bigg).
\end{align*}
\end{lem}

We note that by symmetry, the integral of $h(\nu_1,\nu_2)e(\alpha_1 L \nu_1 + \alpha_2 L \nu_2)$ with respect to the conjugate measure will be similarly bounded. And of course, for $\lambda < (\tfrac{|\alpha_1|}{2}+\tfrac{|\alpha_2|}{2})L$ this estimate is true, but trivially.

Plainly, to prove this we need only show that for $f$ Schwartz on $\mathbb{R}$ with compact Fourier support, and $P \geq 0$,
\begin{equation}
\label{x}
\int_{-\infty}^\infty f(u) e(-Pu) \bigg(-\frac{1}{4\pi^2u^2} + \zeta(1-i2\pi u)\zeta(1+i2\pi u) A(i 2\pi u)\bigg) du = O_f\Big(\frac{1}{e^P}\Big).
\end{equation}

We will use the formula
\begin{prop}
\label{form}
For $f$ Schwartz on $\mathbb{R}$,
\begin{align*}
&\int_{-\infty}^\infty f(u)\Big(-\frac{1}{4\pi^2u^2} + \zeta(1-i2\pi u)\zeta(1+i2\pi u)\Big)\,du \\
&\;= \hat{f}(0) + \int_1^\infty\int_0^\infty (\hat{f}+\hat{f}'')\Big(\log\big(\tfrac{y}{x}\big)\Big)\cdot\big[\mathbf{1}_{[0,1]}(x)\cdot xy - \lfloor y \rfloor (x-\lfloor x\rfloor)\big] \,\frac{dx}{x^2}\frac{dy}{y^2}
\end{align*}
\end{prop}

\begin{proof}[Proof of Proposition \ref{form}]
Note that
$$
\frac{\zeta(s)}{s} = \int_1^\infty \frac{\lfloor x \rfloor}{x^{s+1}} dx \;\;\; \textrm{ for } \Re s > 1,
$$
$$
\frac{\zeta(s)}{s} = -\int_0^\infty \frac{x-\lfloor x \rfloor}{x^{s+1}} dx \;\;\; \textrm{ for } 0 < \Re s < 1,
$$
$$
\frac{1}{s-1} = \int_1^\infty \frac{x}{x^{s+1}} dx \;\;\; \textrm{ for }  \Re s > 1,
$$
$$
\frac{1}{s-1} = -\int_0^1 \frac{x}{x^{s+1}} dx \;\;\; \textrm{ for } 0 < \Re s < 1.
$$
Therefore for $\epsilon > 0$ and $u$ real,
\begin{align*}
&\frac{\zeta(1+\epsilon+i2\pi u)\zeta(1-\epsilon-i 2\pi u)}{1-(\epsilon + i2\pi u)^2}-\frac{1}{-\epsilon-i2\pi u}\frac{1}{\epsilon + i2\pi u}\\
&\;= \int_1^\infty \int_0^\infty \big[\mathbf{1}_{[0,1]}(x)\cdot xy - \lfloor y \rfloor (x-\lfloor x\rfloor)\big] \Big(\frac{x}{y}\Big)^{\epsilon + i2\pi u} \, \frac{dx}{x^2}\frac{dy}{y^2}
\end{align*}
This expression will remain bounded for $u$ near $0$ as $\epsilon \rightarrow 0^+$. If we fix $\epsilon$ and integrate in $u$ with respect to $f(u)du$, the the bottom line becomes
$$
\int_1^\infty\int_0^\infty \hat{f}\Big(\log\big(\tfrac{y}{x}\big)\Big)\cdot\Big(\frac{x}{y}\Big)^\epsilon \big[\mathbf{1}_{[0,1]}(x)\cdot xy - \lfloor y \rfloor (x-\lfloor x\rfloor)\big] \,\frac{dx}{x^2}\frac{dy}{y^2},
$$
where we are justified in changing the order of integration by a trivial application of Fubini's theorem. Now letting $\epsilon \rightarrow 0^+$, we have
\begin{align*}
&\int_{-\infty}^\infty f(u)\bigg(-\frac{1}{4\pi^2u^2} + \frac{\zeta(1-i2\pi u)\zeta(1+i2\pi u)}{1+4\pi^2u^2}\bigg)\,du \\
&\;= \int_1^\infty\int_0^\infty \hat{f}\Big(\log\big(\tfrac{y}{x}\big)\Big)\cdot\big[\mathbf{1}_{[0,1]}(x)\cdot xy - \lfloor y \rfloor (x-\lfloor x\rfloor)\big] \,\frac{dx}{x^2}\frac{dy}{y^2}.
\end{align*}
Replacing $f(u)$ with $(1+4\pi^2u^2)f(u)$ completes the proof.
\end{proof}

\begin{proof}[Proof of Lemma \ref{offdiagonal}]
Returning to Lemma \ref{offdiagonal}, we note that
$$
A(s) = \sum_{n=1}^\infty \frac{\mu(n)}{\phi(n)^2}\sum_{d,\delta | n}\mu(d)\mu(\delta)(d\delta)^{-s}.
$$
For $f$ of compact Fourier support, and $P$ a sufficiently large positive number
\begin{align*}
&\int_{-\infty}^\infty f(u)e(-Pu)\Big(-\frac{1}{4\pi^2u^2} + \zeta(1-i2\pi u)\zeta(1+i2\pi u)\Big)\,du \\
&\;=\hat{f}(P)+ \int_1^\infty\int_0^\infty [\hat{f}+\hat{f}'']\big(\log\big(\tfrac{y}{x}\big)+P\big)\cdot\big[\mathbf{1}_{[0,1]}(x)\cdot xy - \lfloor y \rfloor (x-\lfloor x\rfloor)\big] \,\frac{dx}{x^2}\frac{dy}{y^2}\\
&\;=O_f\bigg(\int_1^\infty \frac{1}{e^p y} \frac{\lfloor y \rfloor}{y^2} dy\bigg)\\
&\;= O_f\Big(\frac{1}{e^P}\Big),
\end{align*}
since, going from the second line to the third, $\log(x) = \log(y) + P + O(1)$, otherwise $[\hat{f}+\hat{f}''](\log(y/x)+P)$ is null. Hence,
\begin{align}
\label{e}
&\int_{-\infty}^\infty f(u)A(i2\pi u)e(-Pu)\Big(-\frac{1}{4\pi^2u^2} + \zeta(1-i2\pi u)\zeta(1+i2\pi u)\Big)\,du\notag\\
&\; \lesssim_f \frac{1}{e^P} \sum_{n=1}^\infty \frac{\mu(n)}{\phi(n)^2}\sum_{d,\delta|n}\frac{\mu(d)\mu(\delta)}{d\delta}\\
&\; \lesssim_f \frac{1}{e^P}\notag
\end{align}

Finally,
\begin{align*}
\int_{-\infty}^\infty f(u) (A(i2\pi u)-1) e(-xu) du  &= \sum_{n > 1} \frac{\mu(n)}{\phi(n)^2}\sum_{d,\delta|n}\mu(d)\mu(\delta)\hat{f}(x+\log(d\delta))\\
&= 0
\end{align*}
for sufficiently large positive $x$. Hence
\begin{align*}
\int_{-\infty}^\infty \frac{f(u)(A(i2\pi u)-1)}{(2\pi u)^2} e(-Pu)\, du &= \int_0^\infty \int_x^\infty \int_{-\infty}^\infty f(u) (A(i2\pi u)-1) e(-Pu-yu)\, du\, dy \,dx\\
& = 0,
\end{align*}
for sufficiently large $P$. Combining this with \eqref{e} proves the lemma.
\end{proof}

These lemmas give us very accurate information about the statistics of $dS(\xi_1)dS(\xi_2)$. Our last two lemmas will allow us to unravel this product into the pair correlation measure we are after.

\begin{lem}
\label{1level}
For $r$ and $\sigma$ with $\hat{r}, \hat{\sigma}$ smooth and compactly supported,
$$
\frac{1}{H}\int_\mathbb{R} \sigma\Big(\tfrac{t-T}{H}\Big) \int_{-\infty}^\infty h\Big(\tfrac{\xi-t}{2\pi}\Big) e\big(\alpha \tfrac{L}{2\pi} (\xi-t)\big) \,dS(\xi) \,dt = O_{h, \sigma}\Big(\frac{1}{H}\Big),
$$
uniformly in $L$.
\end{lem}

\textit{Remark: } This is in essence a 1-level density estimate.

\begin{proof} We have by the explicit formula,
\begin{align*}
&\frac{1}{H}\int_\mathbb{R} \sigma\Big(\tfrac{t-T}{H}\Big) \int_{-\infty}^\infty h\Big(\tfrac{\xi-t}{2\pi}\Big) e\big(\alpha \tfrac{L}{2\pi} (\xi-t)\big) \,dS(\xi) \,dt \\
&\;= \frac{1}{H}\int_\mathbb{R} \sigma\Big(\tfrac{t-T}{H}\Big)\int_{-\infty}^\infty \big[e\big(-\tfrac{xt}{2\pi}\big)\hat{h}(x-\alpha L) + e\big(\tfrac{xt}{2\pi}\big) \hat{h}(-(x+\alpha L))\big] e^{-x/2} \,d\big(e^x-\psi(e^x)\big)\\
&\;\lesssim_h \int_{-\infty}^\infty \hat{\sigma}\big(\tfrac{Hx}{2\pi}\big)e^{-x/2}\,d\big(e^x-\psi(e^x)\big)\\
&\;\lesssim_{h,\sigma}\frac{1}{H}.
\end{align*}
\end{proof}

\begin{lem}
\label{11}
For $r$ and $\sigma$ with $\hat{r}, \hat{\sigma}$ smooth and compactly supported,
\begin{align*}
&\frac{1}{H}\int_\mathbb{R}\sigma\Big(\tfrac{t-T}{H}\Big)\int_{-\infty}^\infty\int_{-\infty}^\infty r(\xi_1-t)r(\xi_2-t)e\big(\alpha_1 \tfrac{L}{2\pi}(\xi_1-t) + \alpha_2 \tfrac{L}{2\pi}(\xi_2-t)\big) dS(\xi_1)\, \frac{\Omega(\xi_2)}{2\pi}\,d\xi_2 \,dt \\
 &\;\;\;\;\;\;\;\;\;= O_{r,\sigma}\Big(\frac{e^{|\tfrac{\alpha_1}{2}|L}}{H}\Big)
\end{align*}
\end{lem}

\begin{proof} By Lemma \ref{Digamma},
\begin{align*}
&\int_{-\infty}^\infty r\Big(\tfrac{\xi_2-t}{2\pi}\Big) e\big(\alpha_2\tfrac{L}{2\pi}(\xi_2-t)\big) \frac{\Omega(\xi_2)}{2\pi} \, d\xi_2\\
&\;=\Big(\tfrac{\Gamma'}{\Gamma}\Big(\tfrac{1}{4}\Big)-\log \pi\Big)\hat{r}(-\alpha_2 L)\\
 &\;\;\;\;\;+ \int_0^\infty \frac{e^{-y/4}}{1-e^{-y}}\big(\hat{r}(-\alpha_2 L) - \tfrac{1}{2}\big[e\big(\tfrac{yt}{2\pi}\big)\hat{r}(-(y+\alpha_2L))+e\big(-\tfrac{yt}{2\pi}\big)\hat{r}(y-\alpha_2L)\big]\big) \,dy.
\end{align*}
We may use Lemma \ref{1level} to deal with the terms attached to $\hat{r}(-\alpha_2 L)$, and employing the explicit formula to deal with the terms that remain,
\begin{align*}
&\frac{1}{H}\int_\mathbb{R}\sigma\Big(\tfrac{t-T}{H}\Big)\int_{-\infty}^\infty\int_{-\infty}^\infty r(\xi_1-t)r(\xi_2-t)e\big(\alpha_1 \tfrac{L}{2\pi}(\xi_1-t) + \alpha_2 \tfrac{L}{2\pi}(\xi_2-t)\big) dS(\xi_1)\, \frac{\Omega(\xi_2)}{2\pi}\,d\xi_2 \,dt \\
 &\;= O_{r,\sigma}\Big(\frac{1}{H}\Big) - \frac{1}{H}\int_\mathbb{R}\sigma\Big(\tfrac{t-T}{H}\Big)\cdot\tfrac{1}{2}\sum_\varepsilon \int_0^\infty \frac{e^{-y/4}}{1-e^{-y}} \int_{-\infty}^\infty e\big(-\tfrac{t}{2\pi}(\varepsilon_1 x + \varepsilon_2 y)\big) \hat{r}(\varepsilon_1 x - \alpha_1 L)\\
  &\;\;\;\;\;\;\;\;\;\;\;\;\;\;\;\;\;\;\;\;\;\;\;\;\;\;\;\;\;\;\;\;\;\;\;\;\; \;\;\;\;\;\;\;\;\;\;\;\;\;\;\;\;\;\;\;\;\;\;\;\;\;\;\;\;\;\;\;\;\;\;\;\;\;\; \;\;\;\;\;\;\;\;\;\;\cdot \hat{r}(\varepsilon_2 y - \alpha_2 L) e^{-x/2} \,d\big(e^x - \psi(e^x)\big) \,dy\,dt \\
 &\;= O\bigg(\bigg|\sum_\varepsilon \int_0^\infty \frac{e^{-y/4}}{1-e^{-y}} \int_{-\infty}^\infty \hat{\sigma}\big(\tfrac{H}{2\pi}(\varepsilon_1 x + \varepsilon_2 y)\big) \hat{r}(\varepsilon_1 x - \alpha_1 L) \hat{r}(\varepsilon_2 y - \alpha_2 L) e^{-x/2} \,d\big(e^x - \psi(e^x)\big) dy\bigg|\bigg)\\
 &\;= O_{r,\sigma}\Big(\frac{e^{|\tfrac{\alpha_1}{2}|L}}{H}\Big).
\end{align*}
\end{proof}

\section{Proof of Theorem \ref{main} and Theorem \ref{4corr}}

\begin{proof}[Proof of Theorem \ref{main}]
We set $\lambda = \log(T/2\pi)/2\pi$. Using Lemmas \ref{6}, \ref{7}, and \ref{offdiagonal} we have
\begin{align}
\label{G}
&\frac{1}{H}\int_\mathbb{R}\sigma\Big(\tfrac{t-T}{H}\Big)\int_{-\infty}^\infty\int_{-\infty}^\infty r\Big(\tfrac{\xi_1-t}{2\pi}\Big)r\Big(\tfrac{\xi_2-t}{2\pi}\Big)e\big(\alpha_1 \tfrac{L}{2\pi}(\xi_1-t)+\alpha_2 \tfrac{L}{2\pi}(\xi_2-t)\big)\,dS(\xi_1)dS(\xi_2)\,dt\notag\\
&\;=O\Big(\frac{T^{|\tfrac{\alpha_1}{2}|+|\tfrac{\alpha_2}{2}|}}{H}\Big)+ \int_{-\infty}^\infty\int_{-\infty}^\infty\Big[\frac{\log(T/2\pi)}{2\pi}\delta(\nu_1-\nu_2) + (2\pi)^2Q_T(2\pi(\nu_1-\nu_2))\Big]\\
&\;\;\;\;\;\;\;\;\;\;\;\;\;\;\;\;\;\;\;\;\; \;\;\;\;\;\;\;\;\;\;\;\;\;\;\;\;\;\;\;\;\;\;\;\; \;\;\;\;\;\;\;\;\; \;\;\;\;\;\;\;\;\;\;\; \;\;\;\;\;\;\cdot r(\nu_1)r(\nu_2)e(\alpha_1 L \nu_1 + \alpha_2 L \nu_2) \,d\nu_1 d\nu_2.\notag
\end{align}
On the other hand,
\begin{align*}
&\Big(\tilde{d}(\xi_1)-\frac{\Omega(\xi_1)}{2\pi}\Big)\Big(\tilde{d}(\xi_2)-\frac{\Omega(\xi_2)}{2\pi}\Big) \\
&\;= \tilde{d}(\xi_1)\tilde{d}(\xi_2) - \Big(\tilde{d}(\xi_1)-\frac{\Omega(\xi_1)}{2\pi}\Big)\cdot\frac{\Omega(\xi_2)}{2\pi} - \Big(\tilde{d}(\xi_2)-\frac{\Omega(\xi_2)}{2\pi}\Big) \cdot \frac{\Omega(\xi_1)}{2\pi} - \frac{\Omega(\xi_1)}{2\pi}\frac{\Omega(\xi_2)}{2\pi},
\end{align*}
so by Lemma 11, the left hand side of \eqref{G} is
\begin{align*}
&O\Big(\frac{T^{|\tfrac{\alpha_1}{2}|}+T^{|\tfrac{\alpha_2}{2}|}}{H}\Big) + \frac{1}{H}\int_\mathbb{R}\sigma\Big(\tfrac{t-T}{H}\Big)\int_{-\infty}^\infty\int_{-\infty}^\infty r\Big(\tfrac{\xi_1-t}{2\pi}\Big)r\Big(\tfrac{\xi_2-t}{2\pi}\Big)e\big(\alpha_1 \tfrac{L}{2\pi}(\xi_1-t)+\alpha_2 \tfrac{L}{2\pi}(\xi_2-t)\big)\\ &\;\;\;\;\;\;\;\;\;\;\;\;\;\;\;\;\;\;\;\;\;\;\;\;\;\;\;\; \;\;\;\;\;\;\;\;\;\;\;\; \;\;\;\;\;\;\;\;\;\; \;\;\;\;\;\;\;\;\;\;\;\; \;\;\;\;\;\;\; \;\;\;\;\;\;\;\;\;\; \;\;\;\;\;\;\; \;\;\;\;\;\;\;\;\cdot \Big(\tilde{d}(\xi_1)\tilde{d}(\xi_2) - \frac{\Omega(\xi_1)}{2\pi}\frac{\Omega(\xi_2)}{2\pi}\Big)\,d\xi_1 d\xi_2\,dt
\end{align*}
However,
\begin{align*}
&\frac{1}{H}\int_\mathbb{R}\sigma\Big(\tfrac{t-T}{H}\Big)\int_{-\infty}^\infty\int_{-\infty}^\infty r\Big(\tfrac{\xi_1-t}{2\pi}\Big)r\Big(\tfrac{\xi_2-t}{2\pi}\Big)e\big(\alpha_1 \tfrac{L}{2\pi}(\xi_1-t)+\alpha_2 \tfrac{L}{2\pi}(\xi_2-t)\big)\tilde{d}(\xi_1)\tilde{d}(\xi_2)\,d\xi_1 d\xi_2\,dt \\
&\;= \frac{1}{H}\int_\mathbb{R}\sigma\Big(\tfrac{t-T}{H}\Big) \sum_{\gamma \neq \gamma'} r\Big(\tfrac{\gamma-t}{2\pi}\Big)r\Big(\tfrac{\gamma'-t}{2\pi}\Big)e\big(\alpha_1\tfrac{L}{2\pi} (\gamma-t) + \alpha_2 \tfrac{L}{2\pi} (\gamma'-t)\big) \,dt \\
&\;\;\;\;\;\;+ \frac{1}{H}\int_\mathbb{R}\sigma\Big(\tfrac{t-T}{H}\Big)\sum_\gamma r^2\Big(\tfrac{\gamma-t}{2\pi}\Big) e\big((\alpha_1+\alpha_2)\tfrac{L}{2\pi}(\gamma-t)\big)\, dt\\
&\;=\frac{1}{H}\int_\mathbb{R}\sigma\Big(\tfrac{t-T}{H}\Big) \sum_{\gamma \neq \gamma'} r\Big(\tfrac{\gamma-t}{2\pi}\Big)r\Big(\tfrac{\gamma'-t}{2\pi}\Big)e\big(\alpha_1\tfrac{L}{2\pi} (\gamma-t) + \alpha_2 \tfrac{L}{2\pi} (\gamma'-t)\big) \,dt\\
&\;\;\;\;\;\;+ O_{r,\sigma}\Big(\frac{1}{H}\Big) + \frac{1}{H}\int_\mathbb{R}\sigma\Big(\tfrac{t-T}{H}\Big)\int_{-\infty}^\infty r^2\Big(\tfrac{\xi-t}{2\pi}\Big) e\big((\alpha_1+\alpha_2)\tfrac{L}{2\pi}(\xi-t)\big)\frac{\Omega(\xi)}{2\pi}\,d\xi\, dt,
\end{align*}
using Lemma \ref{1level} in the last line.

By Stirling's formula, we have that
$$
\Omega(\xi+t+T) = \log(T/2\pi)+O\Big(\frac{1}{|\xi+t+T|+2}\Big) + O\big(\log(1+\tfrac{|\xi+t|}{T})\big)
$$
so that
$$
\frac{1}{H}\int_\mathbb{R} \sigma\Big(\tfrac{t-T}{H}\Big) \frac{\Omega(\xi+t)}{2\pi} dt = \frac{\log (T/2\pi)}{2\pi} + O\Big(\frac{1}{H}\Big) + O\Big(\frac{|\xi|+H}{T}\Big)
$$
and
$$
\frac{1}{H}\int_\mathbb{R} \sigma\Big(\tfrac{t-T}{H}\Big) \frac{\Omega(\xi_1+t)}{2\pi}\frac{\Omega(\xi_1+t)}{2\pi} dt = \Big(\frac{\log(T/2\pi)}{2\pi}\Big)^2 + O\Big(\frac{\log T}{H}\Big) + O\Big(\frac{|\xi_1|+|\xi_2| +H}{T}\Big).
$$
By removing the $\delta$ function, \eqref{G} therefore implies that
\begin{align*}
&\frac{1}{H} \int_\mathbb{R} \sigma\Big(\tfrac{t-T}{H}\Big)\sum_{\gamma\neq \gamma'} r\big(\tfrac{\gamma-t}{2\pi}\big) r\big(\tfrac{\gamma'-t}{2\pi}\big) e\big(\alpha_1\tfrac{L}{2\pi}(\gamma-t)+\alpha_2 \tfrac{L}{2\pi}(\gamma'-t)\big) \,dt + O\Big(\log T\Big(\frac{H}{T}+\frac{1}{H}\Big)\Big) \\
&\;= O\Big(\frac{T^{|\tfrac{\alpha_1}{2}|+|\tfrac{\alpha_2}{2}|}}{H}\Big) \\
&\;\;\;+ \int_{-\infty}^\infty \int_{-\infty}^\infty \bigg[\Big(\frac{\log(T/2\pi)}{2\pi}\Big)^2 + Q_T(\nu_1-\nu_2)\bigg] r\big(\tfrac{\nu_1}{2\pi})r\big(\tfrac{\nu_2}{2\pi}\big)e\big(\alpha_1\tfrac{L}{2\pi}\nu_1+\alpha_2 \tfrac{L}{2\pi}\nu_2\big) \,d\nu_1 d\nu_2,
\end{align*}
as claimed.
\end{proof}

\begin{proof}[Proof of Theorem 4]
In Theorem \ref{main}, we let $h(\nu_1,\nu_2) = \omega(\nu_1-\nu_2)\eta\big(\tfrac{\nu_1+\nu_2}{2}\big)$ where $\eta$ is any function with a smooth and compactly supported Fourier transform and $\hat{\eta}(0) = 1$ say, and we set $\alpha_1=-\alpha_2 = \alpha$. The left hand side of \eqref{mainn} is
$$
\sum_{\gamma\neq \gamma'} \omega(\gamma-\gamma')e\big(\alpha\tfrac{L}{2\pi}(\gamma-\gamma')\big) \frac{1}{H}\int_\mathbb{R}\sigma\Big(\tfrac{t-T}{H}\Big)\eta\big(\tfrac{\gamma+\gamma'}{2}-t\big)\,dt.
$$
On the other hand, averaging $T$ from $0$ to $R$:
$$
\frac{1}{R}\int_0^R \frac{1}{H} \int_\mathbb{R}\sigma\Big(\tfrac{t-T}{H}\Big)\eta(\nu-t) \,dt \,dT = \frac{1}{R} \int_{-\infty}^\infty \eta(\tau) \int_{-\tfrac{R}{H}+ \tfrac{\nu-\tau}{H}}^{\tfrac{\nu-\tau}{H}}\sigma(Q)\,dQ\,dt.
$$
If $H = R^{1-\epsilon/2}$ and we fix $\delta'$ less than $\epsilon/2$, we have that for $\nu \in [R^{1-\delta'}, R-R^{1-\delta'}]$,
$$
\int_{-\infty}^\infty \eta(\tau) \int_{-\tfrac{R}{H}+ \tfrac{\nu-\tau}{H}}^{\tfrac{\nu-\tau}{H}}\sigma(Q)\,dQ\,dt = \int_{-\infty}^\infty \eta(\tau) \,d\tau + O\Big(\frac{1}{R^j}\Big),
$$
for any $j>0$ as both $g$ and $\sigma$ are Schwartz. (The implied constant will vary with $j$.) On the other hand, for $\nu \geq R + R^{1-\delta'}$ or $\nu \leq -R^{1-\delta'}$ this quantity is
$$
O\Big(\frac{1}{R^j}\Big)
$$
for the same reason. For typographical reasons we use the notation
$$
E_1 = \{(\gamma,\gamma')\;: \;\gamma \neq \gamma',\,R^{1-\delta'}\leq \tfrac{\gamma+\gamma'}{2} \leq R - R^{1-\delta'}\},
$$
and
$$
E_2 = \{(\gamma,\gamma')\;:\; \gamma \neq \gamma', \textrm{ and either }R - R^{1-\delta'} \leq \tfrac{\gamma+\gamma'}{2} \leq R + R^{1-\delta'}\textrm{ or } - R^{1-\delta'} \leq \tfrac{\gamma+\gamma'}{2} \leq R^{1-\delta'}\},
$$
and we therefore have,
\begin{align*}
&\frac{1}{R}\int_0^R \sum_{\gamma\neq \gamma'} \omega(\gamma-\gamma')e\big(\alpha\tfrac{L}{2\pi}(\gamma-\gamma')\big) \frac{1}{H}\int_\mathbb{R}\sigma\Big(\tfrac{t-T}{H}\Big)\eta\big(\tfrac{\gamma+\gamma'}{2}-t\big)\,dt\\
&\;=\frac{1}{R}\sum_{(\gamma,\gamma')\in E_1}\omega(\gamma-\gamma')e\big(\alpha\tfrac{L}{2\pi}(\gamma-\gamma')\big)\int_{-\infty}^\infty \eta(\tau) \,d\tau \\
&\;\;\;\;+ O\Bigg(\frac{1}{R}\sum_{(\gamma,\gamma')\in E_2}\omega(\gamma-\gamma')\Bigg) + O\Big(\frac{1}{R^j}\Big)\\
&\;= \frac{1}{R}\sum_{0 < \gamma\neq \gamma'\leq R}\omega(\gamma-\gamma')e\big(\alpha\tfrac{L}{2\pi}(\gamma-\gamma')\big)\int_{-\infty}^\infty \eta(\tau) \,d\tau + O\Big(\frac{\log R}{R^{\delta'}}\Big),
\end{align*}
as $\omega$ is Schwartz and our sums are therefore effectively limited to $\gamma, \gamma'$ with $\gamma-\gamma' = O(1)$. On the other hand, Theorem \ref{main} implies this is
$$
O\Big(\frac{\log R}{R^{\epsilon/2}}\Big) + \int_{-\infty}^\infty\eta(\tau)\,d\tau \int_{-\infty}^\infty \omega(u)e\big(\alpha \tfrac{\log T}{2\pi}u\big) \bigg[\frac{1}{R}\int_0^R \Big(\frac{\log (T/2\pi)}{2\pi}\Big)^2 + Q_T(u)\, dT\bigg]\, du.
$$
Selecting $\delta' > \delta$ gives us the theorem.
\end{proof}

It is worthwhile finally to discuss the difference between this prediction and one in which $Q_t(u)$ has been replaced by
$$
K_t(u) = -\frac{\sin(\pi \tfrac{\log (T/2\pi)}{2\pi} u)}{(\pi u)^2}
$$
This replacement would amount to a na\"{i}ve extension of the GUE conjecture. We want at least to check that
$$
\Delta_1 := \int_{\mathbb{R}}\omega(u)e\big(\alpha \tfrac{\log T}{2\pi}u\big)\bigg[\frac{1}{T}\int_0^T Q_t(u)-K_t(u) \,dt\bigg]\, du
$$
is not negligible, since otherwise the rather more recondite expression involved in defining $Q_t$ would be unnecessary.

In the first place, we showed above, through Lemma \ref{offdiagonal}, that
$$
\Delta_2 := \int_{\mathbb{R}}\omega(u)e\big(\alpha \tfrac{\log T}{2\pi}u\big)\bigg[\frac{1}{T}\int_0^T Q_t(u)-\tilde{Q}_t(u) \,dt\bigg] \,du = O_\delta\big(\frac{1}{T^\delta}\big),
$$
for $\delta$ and $\alpha$ as in Theorem \ref{4corr}, where
$$
\tilde{Q}_t(u) = \frac{1}{4\pi^2}\bigg(\Big(\frac{\zeta'}{\zeta}\Big)'(1+iu) - B(iu) + \Big(\frac{\zeta'}{\zeta}\Big)'(1-iu) - B(-iu) + e(-\tfrac{\log (t/2\pi)}{2\pi} u) + e(\tfrac{\log (t/2\pi)}{2\pi} u)\bigg).
$$
Indeed, if we content ourselves with $\Delta_2$ decaying like $O_k(1/\log^k T)$ for any $k$, instead of $O_\delta(1/T^\delta)$, this is obviously true as long as $\alpha$ is bound away from $-1$ and $1$ -- including $\alpha$ larger than $1$. (And quite false for $\alpha$ either $-1$ or $1$.)

On the other hand,
$$
\tilde{Q}_t(u) - K_t(u) = \frac{1}{4\pi^2}\bigg(2\Re\Big[\Big(\frac{\zeta'}{\zeta}\Big)'(1+iu) - B(iu)\Big]+\frac{2}{u^2}\bigg),
$$
so that for $\alpha$ bound away from $0$
$$
\Delta_3 := \int_{\mathbb{R}}\omega(u)e\big(\alpha \tfrac{\log T}{2\pi}u\big)\bigg[\frac{1}{T}\int_0^T \tilde{Q}_t(u) - K_t(u) \,dt\bigg] \,du = O_k\Big(\frac{1}{\log^k T}\Big),
$$
but for $\alpha = 0$, $\Delta_3$ is not even $o(1)$ as $\tilde{Q}_t(u) - K_t(u)$ has no dependence on $t$ and is not identically $0$ (this would falsely imply that $(\zeta'/\zeta)'(1+iu)$ is bounded, for instance). This is enough to see that $K_t$ cannot replace $Q_t$ in Theorem \ref{main} or \ref{4corr}.

We can examine the difference between the two predictions in greater detail. Note that for $\Re s > 1$,
$$
\sum_{n=1}^\infty \frac{\Lambda^2(n)}{n^s} = \sum_{k=1}^\infty c_k \Big(\frac{\zeta'}{\zeta}\Big)'(ks),
$$
where
$$
c_k = \sum_{d|k}\mu(d)d,
$$
so that
$$
\tilde{Q}_t(u) - K_t(u) = \frac{1}{2\pi^2}\bigg[\Re \Big(\frac{\zeta'}{\zeta}\Big)'(1+iu) + \frac{1}{u^2} + \sum_{k=2}^\infty c_k \Re \Big(\frac{\zeta'}{\zeta}\Big)'(k + iku)\bigg].
$$
By \eqref{zetadiff} the function $\Re \Big(\frac{\zeta'}{\zeta}\Big)'(1+iu) + \frac{1}{u^2}$ has troughs corresponding to each $\gamma$, however these troughs are all of the same width and height and so their appearance is only really striking for low-lying zeroes which are spread far apart. For higher zeroes with less space between them, the troughs interfere with one another. In addition, the functions $\Re \Big(\frac{\zeta'}{\zeta}\Big)'(k + iku)$ have troughs, smaller in depth and width, around $\gamma/k$. Since, for instance, $c_2 = -1$, this corresponds to a small bump in the pair correlation function around the values $\gamma/2$ -- for instance around 14.132/2 = 7.066 -- at least while tested against sufficiently smooth test functions. One can discern this bump in Figure \ref{data}.

We note in passing that for $\alpha > 1+ \epsilon$, the above discussion shows that the Bogomolny-Keating heuristics are consistent with a conjecture that
\begin{align*}
\frac{1}{T} \sum_{0<\gamma\neq \gamma' < T} w(\gamma-\gamma')e\big(\alpha \tfrac{\log T}{2\pi} (\gamma-\gamma')\big) &= o(1) + \int_\mathbb{R} w(u)e\big(\alpha \tfrac{\log T}{2\pi} u\big) \bigg[\frac{1}{T}\int_0^T \Big(\frac{\log(t/2\pi)}{2\pi}\Big)^2 + K_t(u)\,dt\bigg]\,du\\
& = o(1),
\end{align*}
shown by Chan \cite{Chan} to be essentially equivalent to a conjecture of Montgomery and Soundararajan \cite{MontSound} for the second moment of primes in short intervals. Their conjecture was based upon the Hardy-Littlewood conjectures, but even on the assumption of these conjectures, to rigorously extend the domain of $\alpha$ in Theorem \ref{4corr} while maintaining an inverse power-of-$T$ error term will require much more work.

\section{Funding}

Research supported in part by an NSF RTG Grant.

\section{Acknowledgements}

I would like to thank my advisor, Terence Tao, for a number of helpful discussions, Jeffrey Stopple, for corrections, and the anonymous referee for corrections and several useful suggestions. A discussion on the website MathOverflow, available at \url{http://mathoverflow.net/questions/83027/}, was also very helpful to me.

\end{document}